\documentclass[12pt]{amsart}
\usepackage{amsmath,amssymb,enumerate,mathtools,dsfont,bbding}
\newtheorem{theorem}{Theorem}[section]
\newtheorem{claim}{}[theorem]
\newtheorem{lemma}[theorem]{Lemma}

\newtheorem{corollary}[theorem]{Corollary}
\newtheorem{conjecture}[theorem]{Conjecture}
\theoremstyle{definition}

\newcommand{\linestar}{\text{\FiveStarLines}}

\newcommand{\bF}{\mathbb F}

\newcommand{\cE}{\mathcal{E}}

\newcommand{\cM}{\mathcal{M}}
\newcommand{\cN}{\mathcal{N}}

\DeclareMathOperator{\cl}{cl}

\DeclareMathOperator{\PG}{PG}

\DeclareMathOperator{\AG}{AG}
\DeclareMathOperator{\SAG}{AG^{\linestar}}

\newcommand{\del}{ \backslash  }

\numberwithin{subcase}{case}
\numberwithin{subsubcase}{subcase}
\newenvironment{subproof}[1][\proofname]{%
  \begin{proof}[Subproof:]%
}{%
  \end{proof}%
}

\begin{document}

\title{The structure of $I_4$-free and triangle-free binary matroids}
\author[Nelson]{Peter Nelson}
\address{Department of Combinatorics and Optimization, University of Waterloo, Waterloo, Canada. Email address: {\tt apnelson@uwaterloo.ca}}
\author[Nomoto]{Kazuhiro Nomoto}
\address{Department of Combinatorics and Optimization, University of Waterloo, Waterloo, Canada. Email address: {\tt 	knomoto@uwaterloo.ca}}
\thanks{This work was supported by a discovery grant from the Natural Sciences and Engineering Research Council of Canada and an Early Researcher Award from the government of Ontario}
\date{\today}
\begin{abstract}
	A simple binary matroid is called \emph{$I_4$-free} if none of its rank-4 flats are independent sets. These objects can be equivalently defined as the sets $E$ of points in $\PG(n-1,2)$ for which $|E \cap F|$ is not a basis of $F$ for any four-dimensional flat $F$.
	
	We prove a decomposition theorem that exactly determines the structure of all $I_4$-free and triangle-free matroids. In particular, our theorem implies that the $I_4$-free and triangle-free matroids have critical number at most $2$. 
\end{abstract}

\maketitle

\section{Introduction}

This paper proves an exact structure theorem for the simple binary matroids with no four-element independent flat and two-dimensional subgeometry. In this paper we use standard matroid theory terminology, with some minor modifications that we explain below.

A \emph{simple binary matroid} is a pair $M = (E,G)$ where $G$ is a finite binary projective geometry $\PG(n-1,2)$ and $E$, the \emph{ground set}, is any subset of the points of $G$. Abusing notation, we write $G$ for the set of points of $G$. For brevity, we will simply refer to a simple binary matroid as a \emph{matroid} in this paper. Two matroids $M_1=(E_1, G_1)$ and $M_2 = (E_2, G_2)$ are \emph{isomorphic} if there exists an isomorphism from $G_1$ to $G_2$ which maps $E_1$ to $E_2$. We say that a matroid $N$ is an \emph{induced restriction} (or \emph{induced submatroid}) of $M$ if there exists some subgeometry $F$ of $G$ such that $N = (E \cap F, F)$. We use the notation $M \rvert F$ for the matroid $N$. If $M$ has no induced restriction isomorphic to $N$, then we say that $M$ is \emph{$N$-free}.

As opposed to simple binary matroids in the usual sense, our matroids are equipped with an extrinsic ambient space $G$. The \emph{dimension} of $M$ is the dimension of $G$ as a geometry. We do not require the ground set to span $G$; if $E$ does not span $G$, then we say that $M$ is \emph{rank-deficient}, and if $E$ spans $G$, then $M$ is \emph{full-rank}. This definition of matroids allows us to define the complement of a matroid; given a matroid $M = (E,G)$, the \emph{complement} of $M$ is the matroid $M^c = (G \del E, G)$.

We write $I_n$ for the matroid $(B,G)$ where $B$ is a basis of an $n$-dimensional projective geometry $G$. Note that $(E,G)$ is $I_n$-free if and only if $F \cap E$ is not a basis of $F$ for any $n$-dimensional subgeometry $F$ of $G$. A \emph{triangle} of $G$ is a two-dimensional subgeometry. If $E$ contains no triangle of $G$, then $M = (E,G)$ is \emph{triangle-free}. The main result for this paper is a structure theorem for $I_4$-free and triangle-free matroids; such matroids fall into one of two types, each of which arises via simple operations from a basic class of $I_4$-free, triangle-free matroids. 

Suppose that $M = (E,G)$ is an $n$-dimensional matroid. We say that $M$ is a \emph{doubling} if there exists $w \notin E$ for which $E+w = E$. Suppose further that there exists a hyperplane $H$ for which $E \subseteq G \del H$ (such matroids are called \emph{affine}). Then $M' = (E',G')$ is a \emph{$0$-expansion} of $M$ if $G$ is a hyperplane of $G'$ and $E = E'$, and $M' = (E',G')$ is a \emph{$1$-expansion} of $M$ if $G$ is a hyperplane of $G'$ and there exists $x \in G' \del G$ such that $E' = E \cup (x+H) \cup \{x\}$. Note that a $0$-expansion is simply the embedding of the matroid in a larger projective geometry.

We write $\SAG(n-1,2)$ for the $(n+1)$-dimensional matroid $M = (E,G)$ in which there exist nested hyperplanes $H_0 \subseteq G_0 \subseteq G$ and $x \in G \del G_0$, $y \in G_0 \del H_0$ such that $E = (G_0 \del H_0) \triangle \{x,y,x+y\}$. We note that $\SAG(n-1,2)$ is simply the \emph{series extension} of an affine geometry in standard matroid terminology. 

We can now state our main result.

\begin{theorem}\label{main_theorem_intro}
A full-rank matroid $M = (E,G)$ is $I_4$-free and triangle-free if and only if either

	\begin{itemize}
		\item $M$ can be obtained from a $1$-dimensional matroid by a sequence of $0$-expansions and $1$-expansions, or
		\item $M$ can be obtained by a sequence of doublings of $\SAG(n-1,2)$ for some $n \geq 3$. 
	\end{itemize}
\end{theorem}

We remark that the two outcomes of Theorem \ref{main_theorem_intro} are mutually exclusive. For an $n$-dimensional matroid $M = (E,G)$, its \emph{critical number} is $\chi(M) = n - \omega(M^c)$, where $\omega(M)$ is the dimension of a largest subgeometry of $M$ contained in $E$; the critical number of a matroid can be seen as a matroidal analogue of chromatic number for graphs (see [\ref{oxley}] p.588 for a discussion). Then the first outcome in Theorem \ref{main_theorem_intro} results in a class of matroids with critical number $1$ (a matroid with critical number $1$ is called \emph{affine}), whereas the second outcome consists of matroids with critical number $2$. It is also worth noting that the second outcome is a much more restrictive class of matroids; up to isomorphism, there are exactly $n-3$ such matroids for any given dimension $n \geq 4$. 

The proof of Theorem \ref{main_theorem_intro} falls in two parts, based on the existence of an induced $C_5$-restriction. The matroid $C_5$ is the full-rank $4$-dimensional matroid on $5$ elements adding to zero; it will not be difficult to show that an $I_4$-free, triangle-free matroid $M$ is affine if and only if it does not contain an induced $C_5$-restriction. It turns out that having an induced $C_5$-restriction greatly restricts the structure of an $I_4$-free, triangle-free matroid. 

The case when $M$ is affine is more involved, and in fact we will first consider the class of \emph{$AI_4$-free} matroids; we say that a matroid $M = (E,G)$ is $AI_4$-free if for any basis $ \{x_1, x_2, x_3, x_4\} \subseteq E$ of a four-dimensional subgeometry of $G$, there exists some $i \in \{1,2,3,4\}$ such that $\sum_{j \neq i}x_j \in E$. It turns out that if $M = (E,G)$ is $AI_4$-free, then we can always find a special hyperplane $H$ of $G$ such that either $E$ or $G \del E$ is contained in either $H$ or $G \del H$. Once we understand this feature of $AI_4$-free matroids, it will be straightforward to derive the outcome corresponding to the affine matroids in Theorem \ref{main_theorem_intro}. Also, we will provide a structural theorem for $AI_4$-free matroids. 

Finally, we remark that Theorem \ref{main_theorem_intro} settles the case $s=4$ in the following conjecture, made in [\ref{bkknp}]. 

\begin{conjecture}[{[\ref{bkknp}]}]\label{chi_bounded_conj}
For any $s \geq 4$, the class of $I_s$-free and triangle-free matroids has bounded critical number.
\end{conjecture}

\section{Preliminaries}

\subsection*{Flats, cosets and translates} We say that a subgeometry of $G$ is a \emph{flat} of $G$. Viewing $G$ as $\bF_2^n \del \{0\}$, a set $F \subseteq G$ is a flat if and only if $F \cup \{0\}$ is closed under addition. We write $[F]$ for the set $F \cup \{0\}$. By convention, we call flats of dimension 2 \emph{triangles}. A triangle of $G$ is equivalently a triple $\{x, y, x+y \} \subseteq G$ with $x \neq y$. A maximal proper flat of $G$ is a \emph{hyperplane}. 

A \emph{coset} of a flat $F$ in a flat $G \supseteq F$ is any set of the form $A = x + [F]$ for some $x \in G \del F$. We do not consider the set $F$ itself to be a coset; when we wish to include the set $F$ itself, we use the term \emph{translate}. 

\subsection*{Induced restrictions} Recall that if $B$ is a basis of an $n$-dimensional projective geometry $H$, then we write $I_n$ for the matroid $(B,H)$. In this paper, we will frequently be obtaining a contradiction by finding an induced $I_4$-restriction. Therefore, to keep our proofs concise, we will abuse notation and say that the set $B$ itself is an induced $I_n$-restriction. Note also that we will often simply claim that $B$ is an induced $I_4$-restriction rather than explicitly writing out calculations to show that these elements do not span any other points of the matroid; instead enough information will be provided prior to such a claim so that it should be easy to check that $M | \cl(B)$ is indeed an induced $I_4$-restriction. 

\subsection*{Critical number} Recall that the \emph{critical number} $\chi(M)$ of an $n$-dimensional matroid $M = (E,G)$ is $n-k$ where $k$ is the size of a largest projective geometry restriction of $M^c = (G \del E, G)$. If $\chi(M)=1$, then $M$ is \emph{affine}. Below we give a standard characterisation of affine matroids in terms of induced \emph{odd circuits}. For $n$ odd, the \emph{circuit of length $n$}, denoted $C_n$, is the full-rank $(n-1)$-dimensional matroid whose ground set consists of $n$ points that add to zero. Note that odd circuits have critical number exactly $2$. The following characterisation is the matroidal analogue of the characterisation of bipartite graphs in terms of excluded odd cycles. 

\begin{theorem}\label{affine_circuit_characterisation}
A matroid is affine if and only if it has no induced odd circuits.
\end{theorem}

\begin{proof}
The forward direction follows by the observation that the critical number does no increase under induced restrictions.

Conversely, suppose that $M = (E,G)$ contains no induced odd circuits, and let $B \subseteq E$ be a basis of $\cl(E)$. As $M$ has no induced odd circuits, it follows through an inductive argument that $\cl(B+B) \cap E = \varnothing$. But $H = \cl(B+B)$ is a hyperplane, so $M$ is affine.
\end{proof}

Note that odd circuits of length $5$ or more contain an induced $I_4$-restriction. The following easy consequence is used repeatedly throughout the paper, often without explicit reference.

\begin{corollary}\label{affine_I4_tri_free}
If $M$ is $I_4$-free, triangle-free, then $M$ is affine if and only if $M$ does not contain a $C_5$-restriction.
\end{corollary}

\subsection*{Doublings} Recall that a matroid $M=(E,G)$ is a \emph{doubling} if there exists $w \in G \del E$ for $E = w+E$; we sometimes specify such an element $w$ and say that the matroid is a \emph{doubling by $w$}. Note that the condition is equivalent to the existence of $w \in G \del E$ and a hyperplane $H \subseteq G$ not containing $w$ for which $E = [w] + (E \cap H)$. Hence we sometimes specify such a hyperplane $H$ and say that $M$ is the doubling with respect to the matroid $M | H$ by $w$, if this condition holds. The following lemma shows that doublings preserve critical number and the absence of most fixed induced submatroids ([\ref{bkknp}]).

\begin{lemma}[{[\ref{bkknp}], Lemma 2.2}]\label{doubling_prop}
Let $M$ be the doubling of a matroid $M_0$. Then
	\begin{itemize}
		\item $\chi(M) = \chi(M_0)$, and
		\item if $N$ is a matroid that is not a doubling of another matroid, and $M_0$ contains no induced $N$-restriction, then neither does $M$.
	\end{itemize}
\end{lemma}

In particular, the above lemma applies when $N$ is a triangle or $I_n$ for $n \geq 3$. We will sometimes write $D(M)$ to mean the resulting matroid from doubling $M$ and define $D^k(M) = D(D^{k-1}(M))$ recursively for $k \geq 2$.

\subsection*{Expansion operations}

Suppose that $M$ is an affine matroid, so that there exists a hyperplane $H$ for which $E \subseteq G \del H$. Then $M' = (E',G')$ is a \emph{$0$-expansion} of $M$ if $G$ is a hyperplane of $G'$ and $E = E'$. Note that it is simply the embedding of $M$ in an $(n+1)$-dimensional projective geometry. Clearly $0$-expansions preserve critical number and $I_s$-freeness for any $s \geq 1$.

Still supposing that $M$ is affine with a hyperplane $H$ for which $E \subseteq G \del H$, $M' = (E',G')$ is a \emph{$1$-expansion} of $M$ if $G$ is a hyperplane of $G'$ and there exists $x \in G' \del G$ such that $E' = E \cup (x+H) \cup \{x\}$. Note any two $1$-expansions of the same matroid are isomorphic. We have the following easy consequence.

\begin{lemma}\label{affine_exp_prop}
Let $M$ be an affine matroid, and let $M'$ be the $1$-expansion of $M$. Let $s \geq 4$. Then,

\begin{itemize}
	\item $M'$ is affine, and
	\item if $M$ is $I_s$-free, then so is $M'$.
\end{itemize}
\end{lemma}

\begin{proof}
Pick any $y \in G \del H$, and let $H' = \cl(H \cup \{x+y\})$. Then $H'$ is a hyperplane of $G'$ for which $H' \cap E = \varnothing$, so $M'$ is affine. 

Now, suppose towards a contradiction that $M$ is $I_s$-free but $M'$ is not, so that there exist $F=\{v_1, \cdots, v_s\} \subseteq E'$ for which $M' \rvert \cl(F) \cong I_s$. Note that for any distinct three elements $x,y,z \in E' \del E$, we have that $x+y+z \in E'$, so $|E' \cap F| \leq 2$. Hence $|E \cap F| \geq 2$. Let $x,y \in E \cap F$, and pick any arbitrary $z \in E' \del E$. But then $x+y+z \in E'$, a contradiction.
\end{proof}

\subsection*{The matroid $\SAG(n-1,2)$} Let $n \geq 3$. Then $\SAG(n-1,2)$ is the $(n+1)$-dimensional matroid $M = (E,G)$ with nested hyperplanes $H_0 \subseteq G_0 \subseteq G$ and $x \in G \del G_0$, $y \in G_0 \del H_0$ for which $E = G_0 \del H_0 \triangle \{x,y,x+y\}$. In matroid terminology, $\{x, x+y\}$ is a \emph{series pair}; any basis of $E$ must include either $x$ or $x+y$. Moreover, the matroid $\SAG(n-1,2)$ is a \emph{series extension} of $\AG(n-1,2)$; contracting the element $x$ or $x+y$ gives the matroid $\AG(n-1,2)$. Moreover, if $n \geq 4$, then $\{x, x+y\}$ is the unique series pair. Note that $\SAG(n-1,2)$ always contains an induced $C_5$-restriction. It is easy to verify that $\SAG(n-1,2)$ is $I_4$-free, triangle-free, and has critical number exactly $2$.

\subsection*{$AI_4$-freeness} For a matroid $M=(E,G)$, we say that $M$ is \emph{$AI_4$-free} if for any basis $ \{x_1, x_2, x_3, x_4\} \subseteq E$ of a four-dimensional subgeometry of $G$, there exists some $i \in \{1,2,3,4\}$ such that $\sum_{j \neq i}x_j \in E$, It is useful to note that $AI_4$-freeness is preserved under complementation, and this fact will be used without reference.

\begin{lemma}\label{AI4_complement}
A matroid $M=(E,G)$ is $AI_4$-free if and only if $M^c$ is $AI_4$-free.
\end{lemma}

\begin{proof}
Let $B = \{x_1, x_2, x_3, x_4\} \subseteq E$ be an independent set for which $\sum_{j \neq i}x_j \notin E$ for all $ 1 \leq i \leq 4$. Then consider the independent set $B' = \{w_1, w_2, w_3, w_4\}$ of $G \del E$ where $w_i = \sum_{j \neq i}x_j$. Then $\sum_{j \neq i}w_j \in E$ for $1 \leq i \leq 4$.
\end{proof}

\subsection*{Triangle-freeness}

We will use the following two results concerning triangle-freeness. 

\begin{lemma}\label{AI4_tri_free}
If a matroid $M=(E,G)$ is $AI_4$-free and triangle-free, then $M$ is affine.
\end{lemma}

\begin{proof}
As $M$ is $AI_4$-free, it contains no induced odd circuits of length $5$ or more. Therefore $M$ contains no induced odd circuits and is affine by Theorem \ref{affine_circuit_characterisation}.
\end{proof}

\begin{lemma}[{[\ref{bkknp}], Corollary 5.2}]\label{I_3_triangle_free}
If $M=(E,G)$ is both triangle-free and $I_3$-free, then $(E, \cl(E))$ is an affine geometry.
\end{lemma}

\section{The non-affine case}

In this section, we will consider the non-affine $I_4$-free, triangle-free matroids. The goal will be to prove the following.

\begin{theorem}\label{affine_or_cn2}
If $M = (E,G)$ is an $I_4$-free, triangle-free matroid, then either $M$ is affine, or $(E, \cl(E)) \cong D^{k}(\SAG(n-1,2))$ for $n \geq 3$ and $k \geq 0$. 
\end{theorem}

The proof is by induction on $\dim(M)$. We will first prove the following lemma, which describes the case when we can find a hyperplane $H$ such that $M \rvert H$ contains the doubling of an induced $C_5$-restriction. This condition turns out to be very strong, as it implies that the matroid $M$ is also a doubling. 

\begin{lemma}\label{doubled_C5}
Let $M = (E,G)$ be an $I_4$-free, triangle-free matroid with a flat $F$ for which $M \rvert F \cong C_5$. Let $w \in E \del F$. If $M \rvert \cl(F \cup \{w\})$ is the doubling of $M \rvert F$ with respect to $w$, then $M$ is also a doubling with respect to $w$. 
\end{lemma}

\begin{proof}
If $\dim(M) = 5$, then the result follows trivially, so suppose that $\dim(M) > 5$. 

Suppose for a contradiction that $M$ is not a doubling with respect to $w$. This implies that there exists $z \in E \del \cl(F \cup \{w\})$ such that $w + z \notin E$. We will now show that the $6$-dimensional matroid $M \rvert \cl(F \cup \{w, z\})$ contains an induced $I_4$-restriction.

\begin{claim}
Let $\{x_1, x_2, x_3\}$ be any three distinct elements of $F \cap E$. Then $x_1+x_2+x_3+z \notin E$.
\end{claim}

\begin{subproof}
Suppose not, so that $x_1+x_2+x_3+z \in E$; then the set $\{x_1+x_2+x_3+z, w+x_1, w+x_2, w+x_3\}$ is an induced $I_4$-restriction. 
\end{subproof}

We now fix a basis $\{x_1, x_2, x_3, x_4\}$ of $F \cap E$. Let $x_5 = x_1 + x_2 + x_3 + x_4$, so that $F \cap E = \{x_1, x_2, x_3, x_4, x_5\}$. We may apply the above claim with $\{x_1, x_2, x_3\}$, $\{x_1,x_4,x_5\}$, $\{x_2,x_4,x_5\}$ and $\{x_3,x_4,x_5\}$ to obtain that $x_1+x_2+x_3+z, x_2+x_3+z, x_1+x_3+z, x_1+x_2+z \notin E$. Since $x_1+x_2+x_3 \notin E$, this implies that $\{x_1, x_2, x_3, z\}$ is an induced $I_4$-restriction, a contradiction. This completes the proof.
\end{proof}

The next lemma describes the situation in which there is a hyperplane $H$ of $G$ such that $M \rvert H$ is $\SAG(n-1,2)$. This case turns out to be harder and requires a detailed case analysis.

\begin{lemma}\label{series_AG_extend}
Let $M = (E,G)$ be an $n$-dimensional, $I_4$-free, triangle-free matroid, $n \geq 5$. If $G$ has a hyperplane $H$ so that $M \rvert H \cong \SAG(n-3,2)$, then either
\begin{itemize}
	\item $E \subseteq H$,
	\item $M$ is a doubling of $M \rvert H$, or 
	\item $M \cong \SAG(n-2,2)$.
\end{itemize}
\end{lemma}

\begin{proof}
Since $M \rvert H \cong \SAG(n-3,2)$, there exist a hyperplane $F$ of $H$, a hyperplane $F'$ of $F$, and elements $W = \{w_0, w_1\} \subseteq (H \cup E) \del F$ so that $M \rvert H = ( W \cup (y+F'),H)$ where $w_0+w_1=y$. We consider the following cases.

\textbf{Case 1}: There exists $z \in E \del H$ for which $y+z \in E$.

We will make a series of straightforward observations to help understand the structure of $M$.

\begin{claim}\label{I_4_ext_case2_claim1}
\
$(z+F) \cap E = \{y+z\}$
\end{claim}

\begin{subproof}
If $x \in F \del (F' \cup \{y\})$, then clearly $x+z \notin E$ as $\{x,z,x+z\}$ would be a triangle. If there exists $x \in F'$ such that $x+z \in E$, then $\{y+z, x+z, x+y\}$ is a triangle. 
\end{subproof}

\begin{claim}\label{I_4_ext_case2_claim2}
For any triangle $T = \{x_0, x_1, x_0+x_1\} \subseteq F'$ and $w \in W$, $|(z+w+\{x_0+y, x_1+y, x_0+x_1\}) \cap E| > 0$. 
\end{claim}

\begin{subproof}
If not, then $\{w ,z, x_0+y, x_1+y\}$ is an induced $I_4$-restriction. 
\end{subproof}

\begin{claim}\label{I_4_ext_case2_claim3}
For any distinct $x, x' \in F'$, $|\{x+z+w_0 , x'+z+w_1 \} \cap E| <2$.  
\end{claim}

\begin{subproof}
If not, then $\{x+z+w_0, x' +z+w_1, x+x'+y\}$ is a triangle.
\end{subproof}

When $\dim(F') > 2$ we obtain the following observation.

\begin{claim}\label{I_4_ext_case2_claim4}
Provided that $\dim(F') > 2$, if there exists $w \in W$ for which $x_0+w+z \in E$ for some $x_0 \in F'$, then $F'+w+z \subseteq E$.
\end{claim}

\begin{subproof}
Let $x \in F' \del  \{x_0\}$. Since $\dim(F') > 2$, we may select $x_1, x_2 \in F'$ such that $x_0 \notin \cl(\{x_1, x_2\})$ and $x = x_0+x_1+x_2$. Then the set $\{z, x_0 + w+z, x_1+y, x_2+y\}$ is an induced $I_4$-restriction if $x+w+z \notin E$. Therefore $F' + w+z \subseteq E$. 
\end{subproof}

At this point, it is helpful to consider the cases when $\dim(F') > 2$ and $\dim(F')=2$ separately. 

\textbf{Case 1.1: $\dim(F') > 2$}.

We claim that $M$ is a doubling of $M \rvert H$. By \ref{I_4_ext_case2_claim2}, there exists $x \in F'$ and $w \in W$ for which $x+z+w \in E$, and by  \ref{I_4_ext_case2_claim4}, $w+z+F' \subseteq E$. By \ref{I_4_ext_case2_claim3}, $(F'+(y+w)+z) \cap E = \varnothing$. Hence, along with \ref{I_4_ext_case2_claim1}, we have that $E \del H = \{z, y+z\} \cup (w + z + F')$.
Recall that $E \cap H = \{w, w+y\} \cup (y + F')$. Therefore we have
	\begin{align*}
		E &= (E  \cap H) \cup (E \del H) \\
		&=  \{w, w+y\} \cup (y + F') \cup \{z, y+z\} \cup (w + z + F') \\
		&= [y+w+z] + (\{w, w+y\} \cup (y + F') ) = [y+w+z] + (E \cap H).
	\end{align*} 
Since $y+w+z \notin E$, we conclude that $M$ is a doubling of $M \rvert H$. 

\textbf{Case 1.2: $\dim(F') = 2$}.

By \ref{I_4_ext_case2_claim2}, there exists $x \in F'$ such that $x+z+w \in E$ for some $w \in W$. Write $F' = \cl(x,x')$ for $x' \in F'$. Note that $x+z+w \in E$ implies by \ref{I_4_ext_case2_claim3} that $(z+y+w+ F' \del \{x\}) \cap E = \varnothing$. 

If $x+z+y+w \in E$, then \ref{I_4_ext_case2_claim3} gives that $(z+w+ F' \del \{x\}) \cap E = \varnothing$. Hence $E \del H = \{z\} \cup (z+\{y, x+w, x+y+w\})$. Therefore

	\begin{align*}
		E &= (E  \cap H) \cup (E \del H) \\
		&= \{w, y+w\} \cup (y+\{x,x',x+x'\}) \cup \{z\} \cup (z+\{y, x+w, x+y+w\})\\
		&=\{y+x', y+x+x'\} \cup (x+ \cl(y,w+z, x+z)).
	\end{align*} 
	
Since $F'' = \cl(y,w+z, x+z)$ satisfies $F'' \cap E = \varnothing$, it follows that $M \cong \SAG(n-2,2)$

We may therefore assume that $x+z+y+w \notin E$. But then repeated application of \ref{I_4_ext_case2_claim3} implies that $x+w+F' \subseteq E$. The analysis is then identical to Case 1.1. This concludes Case 1.

\textbf{Case 2}: There exist no $z \in G \del H$ for which $\{z, y+z\} \subseteq E$.

Suppose that one of the conclusions, $E \subseteq H$, does not hold. We may select $z \in E \del H$. 

Again, we collect a series of straightforward facts to help understand the structure of $M$. 

\begin{claim}\label{I_4_ext_case1_claim1}
For any triangle $T \subseteq F'$, $T \cap E \neq \varnothing$.
\end{claim}

\begin{subproof}
If not, then $(y+T) \cup \{z\}$ is an induced $I_4$-restriction. 
\end{subproof}

\begin{claim}\label{I_4_ext_case1_claim2}
For any $x \in F'$, $x + z \notin E$ if and only if $(x+z + W) \cap E \neq \varnothing$. Moreover, in this case we have $|(x+z + W) \cap E| = 1$. 
\end{claim}

\begin{subproof}
To show the forward statement, if $(x+z + W) \cap E = \varnothing$, $\{x+y, w_0, w_1, z\}$ is an induced $I_4$-restriction. 

For the reverse direction, note that if $x+z \in E$, then $\{x+z, x+w+z, w\}$ would be a triangle for some $w \in W$.

Finally, since there exists no $t \in G \del H$ for which $\{t, y+t\} \subseteq E$, it follows that $(x+z + W) \cap E \neq \varnothing$ if and only if $|(x+z + W) \cap E| = 1$.
\end{subproof}

\begin{claim}\label{I_4_ext_case1_claim3}
For any distinct $x, x' \in F'$, $|\{x+w_0+z, x'+w_1+z\} \cap  E| < 2$.
\end{claim}

\begin{subproof}
If not, then $\{x+w_0+z, x' +w_1+z, x+x'+y\}$ is a triangle.
\end{subproof}

Let us write $X_0 = \{x \in F' \mid x + w_0 + z \in E\}$, $X_1 = \{x \in F' \mid x + w_1+ z \in E\}$ and $X_2 = \{x \in F' \mid  x+z \in E\}$. By \ref{I_4_ext_case1_claim1}, $X_2 \neq \varnothing$, and \ref{I_4_ext_case1_claim2} implies that $(X_0, X_1, X_2)$ partitions $F'$. Moreover, by  \ref{I_4_ext_case1_claim3}, $|X_0|  \neq 0$ if and only if $|X_1| = 0$. Moreover, when we restrict to a triangle, we have the following.

\begin{claim}\label{I_4_ext_case1_claim4_2}
If $T \subseteq F'$ is a triangle, then
\begin{enumerate}
	\item\label{I_4_ext_case1_claim4_a} if $|T \cap X_2|=3$, or
	\item\label{I_4_ext_case1_claim4_b} if $|T \cap X_2|=1$ and $|T \cap X_i|=2$ for some $i \in \{0,1\}$.
\end{enumerate}
\end{claim}

\begin{subproof}
By \ref{I_4_ext_case1_claim1}, $|T \cap X_2| > 0$. By \ref{I_4_ext_case1_claim2}, $(x+z+W)\cap E = \varnothing$.

Suppose (\ref{I_4_ext_case1_claim4_a}) does not hold, so that there exists $x' \in T$ such that $x'+z \notin E$. By \ref{I_4_ext_case1_claim2}, we know that there exists exactly one $w \in W$ for which $x'+z+w \in E$. 

If $x+x' + z \in E$ holds, then by \ref{I_4_ext_case1_claim2}, $(x+x'+z+W) \cap E = \varnothing$. But then, it follows that $\{x+y, x'+y, x+x'+y, x'+w+z\}$ is an induced $I_4$-restriction. Hence $x+x'+z \notin E$. 

By \ref{I_4_ext_case1_claim2} and \ref{I_4_ext_case1_claim3}, it follows that $x+x'+z+w \in E$ and $x+x'+z+y+w \notin E$, which is (\ref{I_4_ext_case1_claim4_b}).
\end{subproof}

The above claim will be enough to settle the case where $\dim(F')=2$. To handle the case $\dim(F')> 2$, the following observation will be useful. 

\begin{claim}\label{I_4_ext_case1_claim5}
For either $i \in \{0,1\}$, there exist no $3$-dimensional flat $F_0 \subseteq F'$ and a triangle $T \subseteq F_0$ for which $T \subseteq X_2$ and $F_0 \del T \subseteq X_i$.
\end{claim}

\begin{subproof}
Assume for a contradiction that such a triangle $T = \cl({x_1, x_2})$ and a flat $F_0 = \cl(x_0,x_1,x_2)$ exist, $x_0 \notin T$. But then $\{x_0+y, x_1+y, z, x_0+x_2+w_i+z\}$ is an induced $I_4$-restriction. 
\end{subproof}

\textbf{Case 2.1}: $\dim(F')=2$.

We will apply \ref{I_4_ext_case1_claim4_2} to give a case analysis depending on the value of $|F' \cap X_2|$.

If $|F' \cap X_2|=3$, then we have that
	\begin{align*}
		E = W \cup (y + \cl(F' \cup \{y+z\})),
	\end{align*}

so that $M \cong \SAG(n-2,2)$.

If $|F' \cap X_2|=1$ and $|F' \cap X_i|=2$ for $i \in \{0,1\}$, let $x' \in F' \del \{x\}$, and $F'' = \cl(\{x,x'+w_i, y+w_i+z\})$. We then have
	\begin{align*}
		E = \{ w_i, x+y\} \cup ((x+y+w_i) + F''), 
	\end{align*}

where $F'' \cap E = \varnothing$. Hence $M \cong \SAG(n-2,2)$. 

\textbf{Case 2.2}: $\dim(F') > 2$.

We claim that $F' \subseteq X_2$.

We know from \ref{I_4_ext_case1_claim4_2} that $X_2 \neq \varnothing$. Fix $v \in X_2$. 

Suppose towards a contradiction that there exists $v' \in F' \del X_2$, so that $v' \in X_i$ for some $i \in \{0,1\}$

By \ref{I_4_ext_case1_claim4_2} it follows that $v+v' \in X_i$. Since $\dim(F') > 2$, there exists $v'' \notin \cl(\{v,v'\})$. If $v'' \in X_i$, then \ref{I_4_ext_case1_claim4_2} applied to triangles implies that $v+v'+v'' \in X_2$, $v'+v'' \in X_2$ and $v+v'' \in X_i$, but this contradicts \ref{I_4_ext_case1_claim5}. Similarly, if $v'' \in X_2$, then applying \ref{I_4_ext_case1_claim4_2} again implies that $v+v'+v'' \in X_i$, $v' + v'' \in X_i$, and $v+v'' \in X_2$. But then this contradicts \ref{I_4_ext_case1_claim5}. This shows that $F' \subseteq X_2$.

Since $F' \subseteq X_2$, we have $E \del H = (z + [F'])$, so that 
	\begin{align*}
		E &= (E \cap H) \cup (E \del H) \\
		 &= W \cup (y+F') \cup (z + [F']) \\
		 &= W \cup (y + F''),
	\end{align*} 
where $F''=\cl(\{y+z\} \cup F')$ and $F'' \cap E = \varnothing$. Therefore $M \cong \SAG(n-2,2)$. 
\end{proof}

We can now prove the main theorem of this section, restated below.

\begin{theorem}\label{affine_or_cn2_restated}
If $M = (E,G)$ is an $I_4$-free, triangle-free matroid, then either $M$ is affine, or $(E, \cl(E)) \cong D^{k}(\SAG(n-1,2))$ where $n \geq 3$ and $k \geq 0$. 
\end{theorem}

\begin{proof}
We proceed by induction on $\dim(M)$. The cases where $\dim(M)=1,2,3$ are routine to check, so we may assume that $\dim(M) \geq 4$. We may assume without loss of generality that $M$ is full-rank.
 
Suppose that $M$ is not affine, so that it contains $C_5$ as an induced restriction. If $\dim(M) = 4$, then $M \cong C_5 = \SAG(2,2)$, so suppose that $\dim(M) > 4$. By extending a basis of such an induced $C_5$-restriction, we can select a hyperplane $H$ of $G$ such that $M \rvert H$ contains $C_5$ as an induced restriction, and $M \rvert H$ is full-rank. By the inductive hypothesis, we have that $M \rvert H \cong D^{k}(\SAG(n-1,2))$ for some $n \geq 3$ and $k \geq 0$. 

If $k \geq 1$, then $H$ contains the doubling of an induced $C_5$-restriction. By Lemma \ref{doubled_C5}, we have that $M$ is the doubling, say by $w$, of $M \rvert H'$ for some hyperplane $H'$. Note that $M \rvert H'$ is not affine, as the doubling of an affine matroid is always affine. Hence we may apply the inductive hypothesis to $M \rvert H'$ to obtain the required result in this case.

If $k = 0$ then Lemma \ref{series_AG_extend} gives the required result. 
\end{proof}

As an immediate corollary, this shows that the $I_4$-free, triangle-free matroids have critical number at most $2$.

\begin{corollary}\label{cn_at_most_2}
If $M$ is $I_4$-free and triangle-free, then $\chi(M) \leq 2$. 
\end{corollary}

\begin{proof}
By Theorem \ref{affine_or_cn2}, it follows that $\chi(M)=1$, or $(E, \cl(E)) \cong D^{k}(\SAG(n-1,2))$ where $n \geq 3$ and $k \geq 0$. Note that $\SAG(n-1,2)$ has critical number 2 for $n \geq 3$. Since doublings preserve critical number, it follows in the latter case that $\chi(M) = 2$. 
\end{proof}

\section{$AI_4$-freeness}

In order to understand the structure of $I_4$-free, triangle-free matroids with critical number exactly $1$, it is helpful to consider the notion of $AI_4$-freeness and our goal of this section is to prove the following. 

\begin{lemma}\label{AI4_free_special_hyperplane}
If $M = (E,G)$ is $AI_4$-free, then there exists a hyperplane $H$ of $G$ such that either $E$ or $G \del E$ is contained in either $H$ or $G \del H$; that is, either
	\begin{itemize}
		\item $E \subseteq H$ ($M$ is rank-deficient),
		\item $G \del E \subseteq H$ ($M^c$ is rank-deficient),
		\item $E \cap H = \varnothing$, or 
		\item $H \subseteq E$. 
	\end{itemize}
\end{lemma}

We first state an important lemma used repeatedly in this section. Note that translates of $U$ include $U$ itself; when we do not want to include $U$ itself, we consider cosets instead.

\begin{lemma}\label{subset_lemma}
	For every matroid $M = (E,G)$, there is a flat $U$ of $G$ for which $U^c = E + E^c$ and $E$ is a union of translates of $U$. Moreover, if $|E| \ge 2$ then $U \subseteq E+E$, and if $|E| \le |G|-2$ then $U \subseteq E^c + E^c$.
\end{lemma}

\begin{proof}
	For each set $A \subseteq [G]$, let $S_A = \{s \in [G] : s + A = A\}$; this is a subspace of $[G]$, so has cardinality a power of $2$, and since $A$ is a disjoint union of translates of $S_A$ the ratio $|A|/|S_A|$ is an integer. Thus, if $|A|$ is odd, the subspace $S_A$ is trivial. 
	
	Since the statement of the lemma is identical when $E$ is replaced by $E^c$, and $|E|$ is odd, we may assume that $|E|$ is even; thus $|E^c|$ is odd, so $S_{E^c}$ is trivial. We show that the flat $U = S_E \del \{0\}$ satisfies the lemma.
	
	First, let $e + f \in E + E^c$, where $e \in E$ and $f \in E^c$. If $e + f \in [U]$, then $f = (e + f) + e \in (e+f) + E = E$, a contradiction. Therefore $E + E^c \subseteq G \del [U] = U^c$. 
	
	Now, let $a \in U^c$. If $a \in E$, then since $S_{E^c}$ is trivial, we have $a + E^c \ne E^c$, so there is some $f \in E^c$ such that $a + f \in G \del E_c$; since $a \ne f$ we have $a + f \ne 0$ and so $a + f \in E$. Therefore $a \in E + f \subseteq E + E^c$. If $a \in E^c$, then since $a \notin U$, we have $a + E \ne E$ and so there is some $e' \in E$ for which $a + e' \notin E$, Since $a \ne e'$ we have $a + e' \ne 0$ and so $a + e' \in E^c$; this gives $a \in E + E^c$. 
	
	The last two arguments give $U^c = E + E^c$ as required. Since $U = S_E \del \{0\}$, $E$ is a union of some translates of $U$, which also implies that if $|E| \ge 2$ then $U \subseteq E+E$, and if $|E| \le |G|-2$ then $U \subseteq E^c + E^c$. 
\end{proof}

The proof of Lemma \ref{AI4_free_special_hyperplane} will follow from the following lemmas.

\begin{lemma}\label{AI4_extend_hyperplane_1}
Let $M = (E,G)$ be an $AI_4$-free matroid of dimension at least $5$, and $H$ be a hyperplane of $G$ such that $|E \del H| \le 1$. Then $G$ has a hyperplane $H'$ such that either
		$E \subseteq H'$,
		$E \cap H' = \varnothing$, or 
		$H' \subseteq E$.

\end{lemma}

\begin{proof}
We may assume that $M$ is full-rank and that $|E \del H| = 1$, as otherwise the result is trivial. Let $\{v\} = E \del H$. Note that $M \rvert H$ must be full-rank, as otherwise, $r(E) \le r(E \cap H) + 1 \le \dim(G)-1$, so $M$ is not full-rank. Also, we have $\sum_{x \in J} x \in E$ for each three-element linearly independent subset $J$ of $E \cap H$, since otherwise $J \cup \{v\}$ would violate $AI_4$-freeness. In particular, the matroid $M \rvert H$ is $I_3$-free.

Hence, if $M \rvert H$ is triangle-free, it follows from Lemma \ref{I_3_triangle_free} that $M \rvert H$ is an affine geometry, so that $E \cap H = H \del K$ where $K$ is a hyperplane of $H$. Then for any $x \in H \del K$, we have that $E \cap H' = \varnothing$ where $H' = \cl(K \cup \{v+x\})$ is a hyperplane of $G$, giving us the required result. Therefore, we may assume that $M \rvert H$ contains a triangle $T = \{x_1, x_2, x_1+x_2\}$, but then for any $x_3 \in (E \cap H) \del T$ it follows from the above observation that $\{x_1+x_2+x_3, x_1+x_3, x_2+x_3\} \subseteq E$. Since $M \rvert H$ is full-rank, it follows that $H \subseteq E$. 
\end{proof}

\begin{lemma}\label{AI4_extend_hyperplane_2}
Let $M = (E,G)$ be an $AI_4$-free matroid of dimension at least $5$, let $H$ be a hyperplane of $G$, and let $F$ be a hyperplane of $H$ such that $E \cap H \subseteq F$. Then there exists a hyperplane $H'$ of $G$ such that either $E$ or $G \del E$ is contained in either $H'$ or $G \del H'$. 
\end{lemma}

\begin{proof}
Suppose that no such $H'$ exists. In particular, $M$ is full-rank (since otherwise any hyperplane $H'$ containing $\cl(E)$ satisfies $E \subset H'$), and $|E \cap F| \geq 1$ (since otherwise $E \subseteq G \del H$). We will first prove the lemma in the case where $|E \cap F| > 1$.

\textbf{Case 1: $|E \cap F| > 1$}

Let $A_0 = H \del F$, and let $A_1, A_2$ denote the two remaining cosets of $F$ in $G$. By assumption, we have $A_0 \cap E = \varnothing$.

\begin{claim}\label{excluded_patterns_AI4_c1}
	Let $v_1 \in E \cap A_1$ and $v_2 \in E \cap A_2$. If $u,u'$ are distinct elements of $F$ with $u' \in E$, then $|\{u+v_1,u+v_2,u + u'\} \cap E| \ne 1$.
\end{claim}
\begin{subproof}
	Note that $E \cap (F + v_1 + v_2) = E \cap A_0 =  \varnothing$. If $u + u' \in E$ while $u+v_1,u+v_2 \notin E$, then $\{v_1,v_2,u',u+u'\}$ would violate $AI_4$-freeness. If $u + v_i \in E$ for some $i \in \{1,2\}$ while $u+ u',u+ v_{3-i} \notin E$, then $\{u',v_1,v_2,v_i + u\}$ would violate $AI_4$-freeness. 
\end{subproof}

We may assume that $|A_1 \cap E| > 1$; if not, we can take a hyperplane $H' = \cl(F \cup A_2)$, so that $|E \del H'| \leq 1$, and Lemma \ref{AI4_extend_hyperplane_1} gives a contradiction. Similarly, we may assume $|A_2 \cap E| > 1$. If $A_i \cap E \neq A_i$, select $v_i \in A_i \del E$, and otherwise, choose $v_i \in A_i \cap E = A_i$, for $i=1,2$. Note that at most one of $v_1 \in E$ and $v_2 \in E$ holds; otherwise, $A_1, A_2 \subseteq E$, so that choosing $H' = F \cup A_0$ gives $G \del E \subseteq H'$.

Let $X_i = (v_i + (A_i \cap E)) \del \{0\}$ for $i=1,2$. Note that $X_i \subseteq F$ and $|X_i| > 1$ for each $i=1,2$  by our choice of $v_i$. Write $X_0 =  F \cap E$. Note that $|X_0| > 1$.

Now, \ref{excluded_patterns_AI4_c1} implies 
	\begin{align*}
		& (X_0 + X_0)  \cap(X_1+X_1^c) \cap (X_2 + X_2^c) = \varnothing \\
		& (X_0 + X_0^c) \cap (X_1 + X_1)  \cap (X_2 + X_2^c) = \varnothing \\
		&(X_0 + X_0^c) \cap (X_1 + X_1^c) \cap (X_2 + X_2) = \varnothing 
	\end{align*}

We may now apply Lemma \ref{subset_lemma} (with $F$ being the ambient space), to obtain flats $U_i$ for which $X_i + X_i^c = U_i^c$, and $X_i$ is a union of translates of $U_i$ for each $i=0,1,2$ provided $U_i$ is not empty. Moreover, since $|X_i| > 1$, we have $U_{i} \subseteq X_i + X_i$ for $i=0,1,2$.

Hence, for any distinct $i,j,k \in \{1,2,3\}$, we have $$ (X_i + X_i)  \subseteq U_j \cup U_k $$.

Recall that $U_i \subseteq X_i + X_i$ for $i=0,1,2$. Therefore, for any distinct $i,j,k \in  \{1,2,3\}$ we obtain $$U_i \subseteq U_j \cup U_k$$.

Note that each of $U_1, U_2,U_3$ is a flat. This implies that at least two of $U_0, U_1, U_2$ are identical and the third is contained in the other two. Let us write $U_i \subseteq U_j = U_k$ for some $ i,j,k \in \{0,1,2\}$. Let $U = U_j = U_k$. 

Since $U \subseteq X_j + X_j$ and $X_j + X_j \subseteq U_i \cup U = U$, it follows that $U = X_j + X_j$, and similarly $U = X_k + X_k$. We also have $X_i + X_i \subseteq U$. In particular, since $|X_j| > 1$, $U$ is non-empty. Since $U = X_j + X_j$ and $X_j$ is a union of translates of $U$, it follows that $X_j$, and similarly $X_k$, equal a translate of $U$ in $F$. In a similar vein, since $X_i + X_i \subseteq U$, $X_i$ is contained in a translate of $U$ in $F$. To summarise, we have the following.

\begin{claim}\label{U_i_relation}
There exists a non-empty flat $U \subseteq F$ and $j, k \in \{0,1,2\}$, $j \neq k$, such that 
	\begin{itemize}
		\item $X_j$ and $X_k$ equal a translate $U$.
		\item $X_i$ is contained in a translate of $U$.
	\end{itemize}
where $i \in \{0,1,2\} \del \{j,k\}$.
\end{claim}

We say that a set $X$ is full if it equals a translate of $U$. Hence, at least two of $X_0, X_1, X_2$ are full.

We now consider two cases depending on whether $v_1 \in E$ and $v_2 \in E$ (recall that at most one of the two can hold at the same time). In the case $v_1 \in E$ or $v_2 \in E$, we may assume, by symmetry, that $v_2 \in E$ and $v_1 \notin E$. 

\textbf{Case 1.1: $v_1, v_2 \notin E$}

\begin{claim}\label{U_i_relation_2}
At most two of $X_0, X_1, X_2$ are contained in $U$. Moreover, if precisely two of $X_0, X_1, X_2$ are contained in $U$, then $X_0 \subseteq U$.
\end{claim}

\begin{subproof}
Suppose first for a contradiction that $X_i \subseteq U$ for each $i=0,1,2$. If $X_1, X_2$ are all full, then since $|X_0| > 1$, we may select two elements $y_1, y_2 \in X_0$, then $\{y_1,y_2, v_2+y_1+y_2, v_2+y_1+y_2\}$ would violate $AI_4$-freeness by \ref{excluded_patterns_AI4_c1}. If $X_1$ is not full, then we may select $x_1 \in X_1$, and $y_1 \notin X_1$, so that $\{x_1, y_1, x_1+v_1, x_1+y_1+v_2\}$ would violate $AI_4$-freeness by \ref{excluded_patterns_AI4_c1}. Similarly, if $X_2$ is not full, a symmetrical argument shows that it would violate $AI_4$-freeness, giving a contradiction.

Suppose next that precisely two of $X_0, X_1, X_2$ are contained in $U$. For a contradiction, suppose that $X_0$ is not contained in $U$, so that $X_1, X_2 \subseteq U$, and $X_0$ is contained in a coset of $U$. If $X_1, X_2$ are full, then select two elements $y_1, y_2 \in X_0$, then $\{y_1, y_2, v_1+y_1+y_2, v_2+y_1+y_2\}$ would violate $AI_4$-freeness by \ref{excluded_patterns_AI4_c1}. If $X_1$ is not full, then select $x_1 \in X_1$, $y_1 \notin X_1$ and $z_1 \in X_0$. Then $\{x_1+v_1, x_1+y_1+v_2, z_1, x_1+y_1+z_1\}$ would violate $AI_4$-freeness by \ref{excluded_patterns_AI4_c1}. By symmetry, the case where $X_2$ is not full follows, giving a contradiction in all cases.
\end{subproof}

We are now ready to complete the analysis of Case 1.1. In each of the possible outcomes resulting from \ref{U_i_relation_2}, we will show that we can select a hyperplane $H'$ of $G$ that satisfies the theorem or find an induced $I_4$-restriction, giving a contradiction.

Suppose first that none of $X_0, X_1, X_2$ is contained in $U$. Let $X_i \subseteq B_i$ where $B_i$ is a coset of $U$ for $i=0,1,2$. If $B_0 = B_1 = B_2$, then we may assume that there is no other coset of $U$, as otherwise $M$ is rank-deficient. But then, we may select $H' = \cl({U \cup \{v_1, v_2\}})$, and we have $E \subseteq G \del H'$. Therefore we may assume that $B_0, B_1, B_2$ are not identical, so we may assume without loss of generality that $B_0 \neq B_1$. But then $B_1 \cap \cl(E) = \varnothing$ (to see this, note that a general element $z$ of $\cl(E)$ has the form $v_1+v_2+x_0+x_1+x_2+y$ where $x_0 \in B_0 \cup \{0\}$, $x_1 \in B_1 \cup \{v_1\}$, $x_2 \in B_2 \cup \{v_2\}$, $y \in U \cup \{0\}$, and hence $z \notin B_1$ as otherwise it would force $x_1=v_1$, $x_2=v_2$, giving $z \in U \cup B_0$). Therefore $M$ is rank-deficient.

Suppose next that precisely one of $X_0, X_1, X_2$ is contained in $U$. First, suppose that $X_0 \subseteq U$. Then, it follows in a similar way that $B_1 \cap \cl(E) = \varnothing$. Therefore, $M$ is rank-deficient. Hence we may assume without loss of generality that $X_1 \subseteq U$, and $X_0 \subseteq B_0$ and $X_2 \subseteq B_2$ for (possibly identical) cosets $B_0$, $B_2$ of $U$. Note that we may assume that the only cosets are $B_0, B_2, B_0+B_2$, as otherwise $M$ is rank-deficient. Select $x \in B_0$, and let $H' = \cl(U \cup (B_0+B_2) \cup \{v_1+x, v_2\})$. Then we have that $E \subseteq G \del H'$.

Finally, we consider the case where precisely two of $X_0, X_1, X_2$ are contained in $U$. Suppose without loss of generality that $X_0, X_1 \subseteq U$ and $X_2 \subseteq B_2$ where $B_2$ is a coset of $U$. But then it follows that $B_2 \cap \cl(E) = \varnothing$. So $M$ is rank-deficient. 

\textbf{Case 1.2: $v_1 \notin E$, $v_2 \in E$}. 

The fact that $v_2 \in E$ means that $A_2 \subseteq E$ by our choice of $v_2$. Hence $X_2 = F$, and therefore $U_2 = F$. Hence, either $U_0 = F$ or $U_1 = F$. If $U_0 = F$, then choosing the hyperplane $H' = \cl(F \cup A_2)$, we have that $G \del E \subseteq G \del H'$. Hence we may assume that $U_1 = F$. We may also assume that $F \cap E \neq E$, as otherwise $H' = \cl(F \cup A_2)$ satisfies $G \del E \subseteq G \del H'$ again. Let $w_1 \in F \del E$, $w_2 \in F \cap E$. Then $\{v_1+w_1+w_2, w_2, v_2, v_2+w_1+w_2\}$ would violate $AI_4$-freeness by \ref{excluded_patterns_AI4_c1}.

\textbf{Case 2: $|E \cap F| = 1$}

Choose $F''$ to be a hyperplane of $F$ such that $F'' \cap E = \varnothing$. Let $F' = \cl(F'' \cup \{z\})$ for any $z \in A_0$ so that that $F' \cap E = \varnothing$, and consider the three cosets of $F'$ in $G$, denoted $A_0'$, $A_1'$, $A_2'$ where we take $A_0'$ so that $|A_0' \cap E|=1$. Let $v \in A_0' \cap E$. 

If $|A_1' \cap E|>1$ and $A_1' \cap E$ is not full-rank, then we may select a hyperplane $H'$ of $F' \cup A_1'$ such that $E \cap (F' \cup A_1') \subseteq H'$, and $|E \cap H'| > 1$. Case $1$ then applies, and the same holds with $A_2'$. So we may assume without loss of generality that $|A_i' \cap E|=1$ or $A_i' \cap E$ is full rank for each $i=1,2$. 

If $|A_i' \cap E| = 1$ for $i=1,2$, then because $\dim(M) \geq 5$, $M$ is rank-deficient, so we may assume without loss of generality that $A_1' \cap E$ is full-rank. Given three linearly independent vectors $v_1, v_2, v_3 \in A_1' \cap E$, we must have that $v_1+v_2+v_3 \in E$, as otherwise $\{v, v_1, v_2, v_3\}$ would violate $AI_4$-freeness. Therefore, $A_1' \subseteq E$, and let $H' = A_1' \cup F'$. It is then easy to check that the conditions are met to apply Case 1 with the matroid $M^c$ and the hyperplane $H'$ to give the required result.

\end{proof}

We can now prove Lemma \ref{AI4_free_special_hyperplane}. 

\begin{proof}[Proof of Lemma \ref{AI4_free_special_hyperplane}]
Let $M = (E,G)$ be a counterexample of smallest dimension. If $\dim(M)=1,2,3$, then we obtain a contradiction from a routine check, hence we may assume $\dim(M) \geq 4$. 

\begin{claim}
$\dim(M) \geq 5$.
\end{claim}

\begin{proof}
This is a tedious check. If $\dim(M)=4$, then replacing $M$ with $M^c$ if necessary, we may assume that $|E| \leq 7$. We may also assume that $M$ is full-rank. Note that $M$ needs to contain a $C_4$-restriction, on a hyperplane $H$, since $M$ is $AI_4$-free. 

Suppose first that it is an induced $C_4$-restriction. Then $M \rvert H \cong C_4$ and write $E \cap H = \{v_1, v_2, v_3, v_1+v_2+v_3\}$. Let $v_4 \in E \del H$. Then there exists $v \in E \cap H$ such that $v+v_4 \in E$, as otherwise $M^c \rvert \cl(\{v_1+v_2, v_1+v_3, v_1+v_2+v_3+v_4\})$ is an $F_7$-restriction. Without loss of generality, suppose that $v_1+v_4 \in E$. Now, we must have that $v_2+v_3+v_4 \in E$ or $v_1+v_2+v_3+v_4 \in E$, as otherwise, $\{v_2, v_3, v_4, v_1+v_4\}$ would violate $AI_4$-freeness. If $v_2+v_3+v_4 \in E$, then $\{v_1, v_2, v_1+v_4, v_2+v_3+v_4\}$ would violate $AI_4$-freeness. If $v_1+v_2+v_3+v_4 \in E$, then $\{v_1, v_2, v_4, v_1+v_2+v_3+v_4\}$ would violate $AI_4$-freeness.

Hence we may assume that it has no induced $C_4$-restriction. Suppose $|H \cap E| = 6$. Recall that $M$ is full-rank and $|E| \leq 7$. We may take $v_4 \in E \del H$, and $v_1,v_2,v_3 \in E \cap H$ for which $v_1+v_2+v_3 \notin E$, and $\{v_1,v_2,v_3,v_4\}$ would violate $AI_4$-freeness. Hence we may assume $|E \cap H| = 5$. Let $E \cap H = \{v_1,v_2,v_3, v_1+v_2,v_1+v_2+v_3\}$, and pick $v_4 \in E \del H$. By symmetry and the fact that $|E| \leq 7$, we may assume that $v_1+v_4 \notin E$. It follows that $v_1+v_2+v_3+v_4 \in E$, as otherwise $\{v_2,v_3, v_1+v_2,v_4\}$ would violate $AI_4$-freeness. But then $\{v_1, v_2, v_4, v_1+v_2+v_3+v_4\}$ violates $AI_4$-freeness. 
\end{proof}

Hence we have that $\dim(M) \geq 5$. Let $k = \dim(M)$. By minimality, for every hyperplane $H$ of $G$, $H$ contains a hyperplane that satisfies one of the four outcomes. If, for any hyperplane $H$ of $G$, the conditions of Lemma \ref{AI4_extend_hyperplane_2} are satisfied, then Lemma \ref{AI4_extend_hyperplane_2} provides a contradiction. Hence, we may assume that, for every hyperplane $H$ of $G$, either $M \rvert H$ or $M^c \rvert H$ contains a $\PG(k-3,2)$-restriction, the projective geometry of dimension $k-2$. 

Moreover, since $\dim(M)=k \geq 5$, if $M$ contains a $\PG(k-3,2)$-restriction, then $M^c$ cannot contain a $\PG(k-3,2)$-restriction, as otherwise we would have $\dim(M) \geq  2(k-2)$, which implies $k \leq 4$. By switching to $M^c$ if necessary, we may therefore suppose that $M^c$ contains a $\PG(k-3,2)$-restriction in every hyperplane. Now, observe that $M$ is triangle-free, since otherwise any hyperplane containing such a triangle cannot contain a $\PG(k-3,2)$-restriction in $M^c$. Therefore, $M$ is both $AI_4$-free and triangle-free, so by Lemma \ref{AI4_tri_free}, it follows that $M$ is affine.
\end{proof}

The rest of this section describes a structural theorem for $AI_4$-free matroids, which will not be used in the proof of our main theorem Theorem \ref{main_theorem_intro}. For the main theorem, we will only use Lemma \ref{AI4_free_special_hyperplane}. 

In light of Lemma \ref{AI4_free_special_hyperplane}, we define the following four operations for a given $n$-dimensional matroid $M = (E,G)$, which we denote by $\alpha_0$, $\alpha_1$, $\beta_0$ and $\beta_1$. 

\begin{itemize}
	\item $\alpha_0(M)$ is the $0$-expansion of $M$ (recall that the $0$-expansion is an embedding of $M$ in a projective geometry of dimension $n+1$)
	\item $\alpha_1(M) = (E', G')$ is the $(n+1)$-dimensional matroid such that a copy of $G$ is embedded in $G'$, and $E' = E \cup (G' \del G)$
	\item $\beta_0(M) = (E', G')$ is the $(n+1)$-dimensional matroid with a copy of $G$ embedded in $G'$ and $E' = (w+E) \cup \{w\}$ for $w \in G' \del G$
	\item $\beta_1(M) = (E', G')$ is the $(n+1)$-dimensional matroid with a copy of $G$ embedded in $G'$ such that $E' = G \cup (w + E) \cup \{w\}$ for $w \in G' \del G$. 
\end{itemize}

We now state some straightforward facts about these four operations, all of which are easy to verify, and hence the proof is omitted.

\begin{lemma}\label{alpha_beta_technical}
Let $M = (E, G)$ be a matroid. Then the following hold.
	\begin{enumerate}
		\item \label{t1} For $\gamma \in \{\alpha_0$, $\alpha_1$, $\beta_0$, $\beta_1\}$, if $\gamma(M)$ is $AI_4$-free then $M$ is $AI_4$-free.
		\item \label{t2} For $\gamma \in \{\alpha_0, \alpha_1\}$, if $M$ is $AI_4$-free then $\gamma(M)$ is $AI_4$-free.
		\item \label{t3} For $\gamma \in \{\alpha_0, \alpha_1\}$, $M$ is $I_3$-free if and only if $\gamma(M)$ is $I_3$-free.
		\item \label{t4} For $\gamma \in \{\beta_0, \beta_1\}$, if $M$ is $AI_4$-free and $I_3$-free then $\gamma(M)$ is $AI_4$-free. 
		\item \label{t5} For $\gamma \in \{\beta_0, \beta_1\}$, if $M$ contains an induced $I_3$-restriction, then $\gamma(M)$ contains an independent set $\{x_1,x_2,x_3,x_4\} \subseteq E$ for which $\sum_{j \neq i} x_j \notin E$ for all $i$ (i.e., $\gamma(M)$ is not $AI_4$-free).
		\item \label{t6} $\beta_1(M)$ is $I_3$-free.
		\item \label{t7} $\beta_0(M)$ contains an induced $I_3$-restriction unless $E$ is a flat.
		\item \label{t8} If $E$ is a flat, then $\beta_0(M) = \gamma_k \cdots \gamma_1(N_0)$ where $N_0$ is a $1$-dimensional matroid, and $\gamma_i \in \{\alpha_0, \alpha_1\}$ for $1 \leq i \leq k$. 
	\end{enumerate}
\end{lemma}

Using this lemma, we can prove the following structure theorem for $AI_4$-free matroids, stated below.  

\begin{theorem}\label{AI4_structure}
The class of $AI_4$-free matroids is the union of the following two classes. $\cN_0$ denotes the set of $1$-dimensional matroids.

	\begin{enumerate}
		\item $\cM_0 = \{ \gamma_k \cdots \gamma_0 (N_0) \mid k \geq 0, \lambda_i  \in \{\alpha_0, \alpha_1, \beta_1\}, N_0 \in \cN_0 \}$
		\item $\cM_1 = \{ \gamma_k \cdots \gamma_0 \beta_0 (M) \mid k \geq 0, \lambda_i \in \{\alpha_0, \alpha_1\},  M \in \cM_0 \} $
	\end{enumerate}

\end{theorem}

\begin{proof}
First note that all such matroids described are indeed $AI_4$-free by statements \ref{t2}, \ref{t3}, \ref{t4}, \ref{t6} in Lemma \ref{alpha_beta_technical}.

Let $M = (E,G)$ be an $AI_4$-free matroid. By applying Lemma \ref{AI4_free_special_hyperplane} and statement \ref{t1} in Lemma \ref{alpha_beta_technical} iteratively, there is a sequence of operations $\gamma_i$ such that $M = \gamma_k \gamma_{k-1} \cdots \gamma_1 (\cE)$, and $\gamma_i \in \{\alpha_0, \alpha_1, \beta_0, \beta_1\}$ for $i=1,\cdots,k$. Let us write $M_l = (E_l, G_l) = \gamma_{l} \gamma_{l-1} \cdots \gamma_1(\cE)$ and by construction each $M_l$ is $AI_4$-free. 

Now, suppose that $\gamma_l$ is the first occurrence of $\beta_0$, if there is any, so that $\gamma_j \in \{\alpha_0, \alpha_1, \beta_1\}$ for $j=1,2,\cdots,l-1$, and $\gamma_l = \beta_0$. We may assume that $E_{l-1}$ is not a flat, as otherwise, by statement \ref{t8} in Lemma \ref{alpha_beta_technical} we may rewrite $\gamma_l \cdots \gamma_0$ in terms of $\alpha_0$ and $\alpha_1$, and repeat the argument. Now, provided there is such an occurrence of $\beta_0$, statement \ref{t7} in Lemma \ref{alpha_beta_technical} implies that the matroid $M_l$ contains an induced $I_3$-restriction, and therefore we must have $\gamma_i \in \{\alpha_0, \alpha_1\}$ for all $i > l$; otherwise $AI_4$-freeness is violated by statement \ref{t5} in Lemma \ref{alpha_beta_technical} again. The result now follows. 
\end{proof}

\section{The main result}
We are now ready to prove the main structure theorem for $I_4$-free, triangle-free matroids, restated below.

\begin{theorem}\label{main_theorem}
For a full-rank matroid $M = (E,G)$, $M$ is $I_4$-free and triangle-free if and only if 

	\begin{itemize}
		\item $M$ can be obtained by a sequence of $0$-expansions and $1$-expansions of a $1$-dimensional matroid, or
		\item $M$ can be obtained by a sequence of doublings of $\SAG(n-1,2)$, $n \geq 3$. 
	\end{itemize}
\end{theorem}

\begin{proof}
The backward direction follows from Lemmas \ref{doubling_prop}, \ref{affine_exp_prop}.

We now prove the forward direction. By Theorem \ref{affine_or_cn2}, $M$ can either be obtained by a sequence of doublings of $\SAG(n-1,2)$ or $M$ is affine. If the former case holds then we are done, so we may suppose that $M$ is affine, so that there exists a hyperplane $H$ of $G$ for which $E \subseteq G \del H$. 

If $M$ is the empty matroid, then the result is trivially true, so suppose that $E \neq \varnothing$. Pick $z \in E$, and consider the matroid $M_0 = (F, H)$ where $F = \{v \mid v+z \in E\}$. Since $M$ is $I_4$-free, it follows that $M_0$ is $I_3$-free. Moreover, since $M$ is affine, it follows that $M_0$ is $AI_4$-free. 

Let $H'$ be the hyperplane of $H$ from the conclusion of Lemma \ref{AI4_free_special_hyperplane}. We will now go through each conclusion of Lemma \ref{AI4_free_special_hyperplane} to see that in each of the cases, we obtain a $0$-expansion or a $1$-expansion, proving the result. 

\textbf{Case 1: $F \subseteq H'$}.\

In this case, let $H'' = \cl(H' \cup \{z\})$, so that $H''$ is a hyperplane of $G$. Then $M \rvert H''$ is an affine matroid with $H'$ satisfying $E \cap H'' \subseteq H'' \del H'$. Then

	\begin{align*}
		E &= \{z\} \cup (z + F)\\
		& \subseteq \{z\} \cup (z+H') \\
		& \subseteq H''.
	\end{align*}  

Therefore, $M$ is the $0$-expansion of the affine matroid $M \rvert H''$. 

\textbf{Case 2: $H \del F \subseteq H'$}.\

Let $H'' = \cl(H' \cup \{z\})$ as before. Let $w \in H \del H'$. Then

	\begin{align*}
		E &= \{z\} \cup (z + F)\\
		& = \{z\} \cup (z+F \cap H')\cup(z + F \del H') \\
		& = \{z\} \cup (z+F \cap H')\cup(z + H \del H') \\
		& = (E \cap H'') \cup \{z+w\} \cup (z+w+H')
	\end{align*}  

Therefore, $M$ is the $1$-expansion of the affine matroid $M \rvert H''$.

\textbf{Case 3: $F \cap H' = \varnothing$}. 

Note that if $M_0$ is rank-deficient, then we are in Case 1, so assume that $M_0$ is full-rank. Observe that $M_0$ is $I_3$-free and triangle-free (since it is affine). Therefore, Lemma \ref{I_3_triangle_free} implies that $M_0$ is a full-rank affine geometry. We are in Case 2. 

\textbf{Case 4: $H' \subseteq F$}. 

Let $w \in F \del H'$ (if no such $w$ exists, then $M_0$ is rank-deficient and we are in Case 1), and let $H'' = \cl(H' \cup \{z+w\})$. 

Then $M \rvert H''$ is an affine matroid with $H'$ as its hyperplane such that $H'' \cap E \subseteq H'' \del H'$. Then

	\begin{align*}
		E &= \{z\} \cup (z + F)\\
		& = \{z\} \cup (z+F \cap H')\cup(z + F \del H') \\
		& =  \{z\} \cup \{z + H'\} \cup  \{z + w\} \cup  (z+w+ F \cap H')  \\
		& =  \{z\} \cup \{z + H'\} \cup (E \cap H'')
	\end{align*} 

Therefore, $M$ is the $1$-expansion of the affine matroid $M \rvert H''$. 

Thus $M$ is either the $0$-expansion or $1$-expansion of another affine matroid of smaller dimension. The result now follows by induction on $\dim(M)$. 
\end{proof}

\section*{References}
\newcounter{refs}
\begin{list}{[\arabic{refs}]}
{\usecounter{refs}\setlength{\leftmargin}{10mm}\setlength{\itemsep}{0mm}}

\item\label{bb}
R. C. Bose, R. C. Burton, 
A characterization of flat spaces in a finite geometry and the uniqueness of the Hamming and the MacDonald codes, 
J. Combin. Theory 1 (1966), 96--104. 

\item\label{bkknp}
M. Bonamy, F. Kardo\v{s}, T. Kelly, P. Nelson, L. Postle,
The structure of binary matroids with no induced claw or Fano plane restriction,
Advances in Combinatorics, 2019:1,17 pp.

\item\label{bhw}
A. A. Bruen, L. Haddad, D. Wehlau,
Binary Codes and Caps,
J. Combin. Des 6 (1998) 275-284.

\item\label{dt}
A. A. Davydov, L. M. Tombak,
Quasiperfect linear binary codes with distance 4 and complete caps in projective geometry,
Problems of Information Transmission 25 No. 4 (1990), 265-275.

\item\label{crit_threshold}
J. Geelen, P. Nelson,
The critical number of dense triangle-free binary matroids,
J. Combin. Theory Ser. B 116 (2016), 238--249.

\item\label{g85}
A. Gy\'arf\'as, Problems from the world surrounding perfect graphs, Proceedings of the International Conference on Combinatorial Analysis and its Applications, (Pokrzywna, 1985), Zastos. Mat. 19 (1987), 413--441.

\item\label{nn}
P. Nelson, K. Nomoto,
The structure of claw-free binary matroids,
arXiv:1807.11543 (2018).

\item \label{oxley}
J. G. Oxley, 
Matroid Theory,
Oxford University Press, New York (2011).

\item\label{s81}
D.P. Sumner, 
Subtrees of a graph and chromatic number, in The Theory and Applications of Graphs, (G. Chartrand, ed.), John Wiley \& Sons, New York (1981), 557--576.

\end{list}

\end{document}